\documentclass[12pt,a4paper]{amsart}
\usepackage{mathrsfs}

\newtheorem{theo+}              {Theorem}           [section]
\newtheorem{prop+}  [theo+]     {Proposition}
\newtheorem{coro+}  [theo+]     {Corollary}
\newtheorem{lemm+}  [theo+]     {Lemma}
\newtheorem{exam+}  [theo+]     {Example}
\newtheorem{rema+}  [theo+]     {Remark}
\newtheorem{defi+}  [theo+]     {Definition}

\newenvironment{theorem}{\begin{theo+}}{\end{theo+}}

\newenvironment{corollary}{\begin{coro+}}{\end{coro+}}

\usepackage{amsthm}
\theoremstyle{plain} \theoremstyle{remark}

\newtheorem{example}{Example}

\def \r{\mbox{${\mathbb R}$}}
\def\E{/\kern-1.0em \equiv }

\def\fb{$f$-biharmonic\;}

\title{$f$-Biharmonic maps and $f$-biharmonic submanifolds II}
\author{Ye-Lin Ou$^{*}$ }
\thanks{$^{*}$ Research supported by
Faculty Development Program of Texas A $\&$ M University-Commerce, 2015.}
\address{Department of
Mathematics,\newline\indent Texas A $\&$ M University-Commerce,
\newline\indent Commerce, TX 75429,\newline\indent USA.\newline\indent
E-mail:yelin.ou@tamuc.edu }
\begin{document}

\title[$f$-Biharmonic maps and $f$-biharmonic submanifolds II]
{$f$-Biharmonic maps and $f$-biharmonic submanifolds II }

\thanks{}
\date{04/28/2016}

\maketitle
\section*{Abstract}
\begin{quote} We continue our study \cite{Ou4} of \fb maps and 
\fb \\submanifolds  by exploring the applications of \fb maps and the relationships among biharmonicity, $f$-biharmonicity and conformality of maps between Riemannian manifolds. We are able to characterize harmonic maps and minimal submanifolds by using the concept of \fb maps and prove that the set of all \fb maps from $2$-dimensional domain is invariant under the conformal change of the metric on the domain. We  give an improved equation for  \fb  hypersurfaces and use it to prove some rigidity theorems about \fb hypersurfaces in  nonpositively curved manifolds, and to give some classifications of \fb hypersurfaces in Einstein spaces and in space forms. Finally, we also use the improved  \fb hypersurface equation to obtain an improved equation and some classifications of biharmonic conformal immersions of surfaces into a $3$-manifold.
{\footnotesize  }
\end{quote}

\section{Introduction}
 {\em Biharmonic maps} and  {\em \fb maps}  are maps $\phi: (M,g)\longrightarrow (N,h)$ between Riemannian manifolds which are the critical points of 
\begin{eqnarray}\notag
{\rm the\; bienergy:}\;\;\;\;E_{2}(\phi)=\frac{1}{2}\int_\Omega |\tau(\phi)|^2v_g,\;\;\;{\rm and}\\\notag
{\rm the\;} f{\rm-bienergy:}\;\;\;E_{2,f}(\phi)=\frac{1}{2}\int_\Omega f\,|\tau(\phi)|^2v_g,
\end{eqnarray}
respectively, where $\Omega$ is a compact domain of $ M$ and $\tau(\phi)$ is the tension field of the map $\phi$.
The Euler-Lagrange equations of these functionals give the biharmonic map equation (see \cite{Ji1})
\begin{equation}\label{BTF}
\tau_{2}(\phi):={\rm
Trace}_{g}(\nabla^{\phi}\nabla^{\phi}-\nabla^{\phi}_{\nabla^{M}})\tau(\phi)
- {\rm Trace}_{g} R^{N}({\rm d}\phi, \tau(\phi)){\rm d}\phi =0,
\end{equation}
where $R^{N}$ denotes the curvature operator of $(N, h)$ defined by
$$R^{N}(X,Y)Z=
[\nabla^{N}_{X},\nabla^{N}_{Y}]Z-\nabla^{N}_{[X,Y]}Z,$$

and the $f$-biharmonic map equation  ( see\cite{Lu} and also \cite{Ou4})
\begin{equation}\label{fbhe}
\tau_{2,f} (\phi)\equiv f\tau_2(\phi)+(\Delta f)\tau(\phi)+2\nabla^{\phi}_{{\rm grad}\, f}\tau(\phi)=0,
\end{equation}
where  $\tau_2( \phi)$ is the  bitension field of $\phi$ defined in (\ref{BTF}).\\

 As the images of biharmonic isometric immersions, biharmonic submanifolds have become a popular topic of research  with many progresses since the beginning of this century (see a recent survey \cite{Ou5} and the references therein for more details). In a recent paper \cite{Ou4}, we introduced and studied $f$-biharmonic submanifolds as those submanifolds whose isometric immersions $\phi: M^m\longrightarrow (N^n, h)$ are $f$-biharmonic maps for some positive function $f$ defined on the submanifolds.  $f$-Biharmonic submanifolds are generalizations of of biharmonic submanifolds and of minimal submanifolds since (i) \fb submanifolds with $f={\rm constant}$ are precisely biharmonic submanifolds, and (ii) any minimal submanifold is $f$-biharmonic for any positive function defined on the submanifold. As it is seen in \cite{Ou4}  and Example 1 of Section \ref{Sect3} of this paper, there are many examples of $f$-biharmonic submanifolds which are neither biharmonic nor minimal. \\

An interesting link among biharmonicity, $f$-biharmonicity and conformality is the following theorem which, although holds only for two-dimensional domain, has played an important role in the study of  biharmonic maps from  two-dimensional domains (see \cite{Ou4} and \cite{WOY}) and will be used to study  biharmonic conformal immersions in the last section of this paper.\\

{\bf Theorem A.}(\cite{Ou4} A map $\phi: (M^2, g)\longrightarrow (N^n, h)$ is an $f$-biharmonic map if and only if $\phi:
(M^2, f^{-1}g)\longrightarrow (N^n,h)$ is a biharmonic map. In particular, if a surface (i.e., an isometric immersion)  $\phi: (M^2, g=\phi^*h)\longrightarrow (N^n, h)$ is an $f$-biharmonic surface if and only if the conformal immersion$\phi:(M^2, f^{-1}g)\longrightarrow (N^n,h)$ is a biharmonic map.\\\

Regarding the topological and the geometric obstacles to the existence of general \fb maps, it was proved in \cite{Ou4} that an \fb map from a compact manifold into a non-positively curved manifold with constant $f$-bienergy density is a harmonic map, and that any \fb map from a compact manifold into Euclidean space is a constant map. \\

In this paper, we continue our work  on \fb maps and \fb \\submanifolds  by further exploring the applications of \fb maps and the relationships among harmonicity, biharmonicity, $f$-biharmonicity and conformality of maps between Riemannian manifolds. We prove that a map between Riemannian manifold is a harmonic map if and only if it is an \fb map for any positive function $f$, and that a submanifold of a Riemannian manifold is minimal if and only if it is \fb for any positive  function $f$ defined on the submanifold. We show that the set of all \fb maps from $2$-dimensional domain is invariant under the conformal change of the metric on the domain. We prove some rigidity theorems about \fb hypersurfaces in  nonpositively curved manifolds and give some classifications of \fb hypersurfaces in Einstein spaces and in space forms by using an improved form of  \fb  \\hypersurface equation. We also use the improved  \fb hypersurface equation to obtain an improved  equation and some classifications of biharmonic conformal immersions of surfaces into a $3$-manifold.

\section{Harmonicity, biharmonicity, $f$-biharmonicity and conformality}

It is well known that harmonicity of maps from a $2$-dimensional domain is invariant under conformal transformations of the domain. It is also known that biharmonic maps do not have this kind of invariance regardless of domain dimensions. However,  as it is shown by our next theorem, the set
\begin{eqnarray}\notag
\mathcal{H}_{2, f}(M^2, N^n)=&&\{(\phi, f)|\; \phi: (M^2, g)\longrightarrow (N^n,h),\; f: (M^2, g)\longrightarrow (0,\infty),\\ \notag && \phi \;{\rm is\;an\;f}-{\rm biharmonic\;map}\}
\end{eqnarray}
 of all \fb maps is invariant under conformal transformation of the domain.
 \begin{theorem}
 The set $\mathcal{H}_{2, f}(M^2, N^n)$ of all \fb maps from a surface into a manifold is invariant under conformal transformation of the domain, more precisely, if $ \phi: (M^2, g)\longrightarrow (N^n,h)$ is \fb for a function $f=\alpha$, then $ \phi: (M^2, {\bar g}=\lambda^2 g)\longrightarrow (N^n,h)$ also an \fb map for $f=\alpha \lambda^{2}$.
\end{theorem}
\begin{proof}
Let $ \phi: (M^2, g)\longrightarrow (N^n,h)$ be an \fb map for a function $f=\alpha$, then by Theorem A, the map $ \phi: (M^2, \alpha^{-1}g)\longrightarrow (N^n,h)$ is a biharmonic map. Noting that $(M^2, \alpha^{-1}g)=(M^2, (\alpha^{-1}\lambda^{-2})\lambda^2g)=(M^2, (\alpha \lambda^{2})^{-1}{\bar g})$, we conclude that the map $ \phi: (M^2, (\alpha \lambda^{2})^{-1}{\bar g})\longrightarrow (N^n,h)$ is a biharmonic map. Using Theorem A again we conclude that the map $ \phi: (M^2, {\bar g}=\lambda^{2}g)\longrightarrow (N^n,h)$ is an \fb map with $f=\alpha \lambda^{2}$. Thus, we obtain the theorem.
\end{proof}

It is clear from the definition of  \fb maps that a harmonic map $\phi:(M^m, g)\longrightarrow (N^n,h)$ is an $f$-biharmonic map for any function $f:(M^m, g)\longrightarrow (0, \infty)$.  Our next theorem shows that this property can actually be used to characterize harmonic maps between Riemannian manifolds.
\begin{theorem}
A map $\phi:(M^m, g)\longrightarrow (N^n,h)$ between Riemannian manifold is a harmonic map if and only if it is an $f$-biharmonic map for any function $f:(M^m, g)\longrightarrow (0, \infty)$.
\end{theorem}
 \begin{proof}
 We only need to prove the ``if" part of the statement, i.e.,  if a map $\phi:(M^m, g)\longrightarrow (N^n,h)$ is an $f$-biharmonic map for any function $f:(M^m, g)\longrightarrow (0, \infty)$, then it is a harmonic map. To prove this is equivalent to show that if the equation
  \begin{equation}\label{fbihe}
 f\tau_2(\phi)+(\Delta f)\tau(\phi)+2\nabla^{\phi}_{{\rm grad}\, f}\tau(\phi)=0
 \end{equation}
 holds for all positive functions $f$ on $(M^m, g)$, then we have $\tau(\phi)\equiv 0$.
 
In fact, take a positive constant function $f=C$, then equation (\ref{fbihe}) implies that the map $\phi$ is  a biharmonic map, i.e., $\tau_2(\phi)\equiv 0$.
 It follows that if a map $\phi:(M^m, g)\longrightarrow (N^n,h)$ is an $f$-biharmonic maps for any function $f$, then we have
  \begin{equation}\label{A2}
(\Delta f)\tau(\phi)+2\nabla^{\phi}_{{\rm grad}\, f}\tau(\phi)=0 \;\;\forall\;f:(M^m, g)\longrightarrow (0, \infty).
 \end{equation}
 If there exists a point $p\in M$ such that $\tau(\phi)(p) \ne 0$, then there is an open neighborhood $U$ of $p$ so that $\tau(\phi)(q) \ne 0,\; \forall\; q\in U$. By taking dot product of both sides of (\ref{A2}) with $\tau(\phi)$,  we have, on $U$,
 \begin{equation}\label{A3}
 (\Delta f)|\tau(\phi)|^2+({\rm grad}\, f)|\tau(\phi)|^2=0 \;\;\forall\;f:(U, g|_U)\longrightarrow (0, \infty).
 \end{equation}
It is well known that  there exists a local harmonic coordinate system in a neighborhood of any point of a Riemannian manifold. By this well known fact and a translation in Euclidean spaces, we can choose a neighborhood $V\subset U$ of $p$ with harmonic coordinate $\{x^i\}$ satisfying $x^i>0$ for all $i=1, 2, \cdots, m$. By choosing positive harmonic coordinate functions $f=x^i$ and substituting it into (\ref{A3})  we have 
\begin{equation}
g^{ik}\frac{\partial}{\partial x^k} |\tau(\phi)|^2=0, \;\; \forall\; i=1, 2, \cdots, m.
\end{equation}

It follows that $|\tau(\phi)|^2={\rm constant}$, and hence Equation (\ref{A3}) reduces to
  \begin{equation}\label{A4}
 (\Delta f)|\tau(\phi)|^2=0 \;\;\forall\;f:(V, V|g)\longrightarrow (0, \infty).
 \end{equation}
 Now let $f=(x^1)^2$, where $x^1$ is the first coordinate function, and substitute it into Equation (\ref{A4}) to have $ (\Delta f)|\tau(\phi)|^2=2g^{11}|\tau(\phi)|^2=0$ on $V$. Since the Riemannian metric is positive definite, we conclude that $|\tau(\phi)|^2=0$ on $V$, which contradicts the assumption that $\tau(\phi)\ne 0$ on $U$. The contradiction shows that we must have $\tau(\phi)\equiv 0$ on $M$,  i.e., the map $\phi$ is a harmonic map. This completes the proof of the theorem.
\end{proof}

As an immediate consequence, we have the following characterization of minimal submanifolds using the concept of \fb submanifolds.
\begin{corollary}\label{23}
A submanifold $M^m\hookrightarrow (N^n, h)$ is minimal if and only if it is \fb for any positive function defined on the submanifold.
\end{corollary}

To close this section, we take a look at maps which are \fb for two different functions, especially those maps which are both biharmonic (i.e., \fb for $f=$constant) and \fb for $f\ne constant $.\\

It is shown in \cite{Ou4} that conformal immersions $\phi: \r^4\setminus\{0\}\longrightarrow \r^4$ given by the inversion of the sphere $S^3$, $\phi(x)=\frac{x}{|x|^2}$ is a biharmonic map and also an \fb map with $f(x)=|x|^4$. Our next theorem shows that this situation cannot happen for pseudo-umbilical isometric immersions.
\begin{theorem}\label{MT1}
If a pseudo-umbilical proper biharmonic submanifold $\phi: M^m\longrightarrow (N^n,h)$ is \fb, then $f$ is constant.
\end{theorem}
\begin{proof}
$\phi: M^m\longrightarrow (N^n,h)$ be a pseudo-umbilical submanifold, then, by definition, its shape operator with respect to the mean curvature vector field $\eta$ is given by
\begin{equation}
A_{\eta}(X)=\langle \eta, \eta\rangle X
\end{equation}
for  any vector field  $X$ tangent to the submanifold. It is well known that the tension field of the submanifold is given by $\tau(\phi)=m\eta$. If the submanifold is both biharmonic and \fb, then we have both $\tau_2(\phi)=0$ and $\tau_{2,f} (\phi)=0$ identically. Substituting these into Equation (\ref{fbhe}) we obtain
\begin{eqnarray}\label{GD1}
0 &=& (\Delta f)\tau(\phi)+2\nabla^{N}_{{\rm grad}\, f}\tau(\phi)\\\notag
&=& m [(\Delta f)\eta+2\nabla^{N}_{{\rm grad}\, f}\eta]\\\notag
&=& m [(\Delta f)\eta+2\nabla^{\perp}_{{\rm grad}\, f}\eta-2A_{\eta}(\rm{grad}\,f)]\\\notag
&=& m [(\Delta f)\eta+2\nabla^{\perp}_{{\rm grad}\, f}\eta-2\langle \eta, \eta\rangle\,\rm{grad}\,f],
\end{eqnarray}
where the third equality was obtained by using the Weingarten formula for submanifolds and the last equality was obtained by the assumption that the submanifold is pseudo-umbilical. Comparing the normal and the tangent components of the both sides of (\ref{GD1}) we have $\rm{grad}\,f=0$ since the submanifold is non-minimal, and hence $f=\rm constant$. This completes the proof of the theorem.
\end{proof}
\begin{corollary}
 Biharmonic hypersurface $\phi: S^m(\frac{1}{\sqrt{2}})\hookrightarrow S^{m+1},\; \phi(x)=(\frac{1}{\sqrt{2}}, x)$ is an \fb hypersurface if and only if $f$ is constant. In particular, there is no nonconstant function $f$ on $S^2(\frac{1}{\sqrt{2}})$ so that the conformal immersion $\phi: (S^2(\frac{1}{\sqrt{2}}), f^{-1}g_0)\longrightarrow (S^3, h_0),\; \phi(x)=(\frac{1}{\sqrt{2}}, x)$ becomes a biharmonic map, where $g_0$ and $h_0$ are the metrics on $S^2(\frac{1}{\sqrt{2}})$ and $S^3$ induced from their ambient Euclidean spaces respectively.
\end{corollary}
\begin{proof}
The first statement of the corollary follows from Theorem \ref{MT1} and the fact that the hypersurface $\phi: S^m(\frac{1}{\sqrt{2}})\hookrightarrow S^{m+1},\; \phi(x)=(\frac{1}{\sqrt{2}}, x)$ is a totally umbilical biharmonic hypersurface of $S^{m+1}$ (See \cite{CMO1}). The second statement of the corollary follows from the first statement and Theorem A.
\end{proof}
\section{ The equations  of \fb hypersurfaces and some classification results }\label{Sect3}
In this section, we will give an improved equation for \fb hypersurfaces in a general Riemannian manifold and use it to prove some rigidity theorems and  some classifications of \fb hypersurfaces in Einstein spaces, including space forms.
\begin{theorem}\label{FBH}
A hypersurface $\phi:M^{m}\longrightarrow (N^{m+1}, h)$  with the mean curvature vector field $\eta=H\xi$ is  $f$-biharmonic if and only if:
\begin{equation}\label{fbh2}
\begin{cases}
\Delta (f\,H)-(f\,H )[|A|^{2}-{\rm
Ric}^N(\xi,\xi)]=0,\\
A\,({\rm grad}(f\,H)) +(f\,H)[\frac{m}{2} {\rm grad}\, H
-({\rm Ric}^N\,(\xi))^{\top}]=0,
\end{cases}
\end{equation}
where ${\rm Ric}^N : T_qN\longrightarrow T_qN$ denotes the Ricci
operator of the ambient space defined by $\langle {\rm Ric}^N\, (Z),
W\rangle={\rm Ric}^N (Z, W)$, $A$ is the shape operator of the
hypersurface with respect to the unit normal vector $\xi$, and $\Delta, \rm grad$ are the Laplace and the gradient operator of the hypersurface respectively.
\end{theorem}
\begin{proof}
It was proved in \cite{Ou4} that an isometric immersion $\phi:M^{m}\longrightarrow N^{m+1}$
 with mean curvature vector $\eta=H\xi$ is $f$-biharmonic if and only if:
\begin{equation}\label{fbh}
\begin{cases}
\Delta H-H |A|^{2}+H{\rm
Ric}^N(\xi,\xi)+H(\Delta f)/f+2({\rm grad} \ln f) H=0,\\
A\,({\rm grad}\,H) +\frac{m}{2}H {\rm grad}\, H
-\, H \,({\rm Ric}^N\,(\xi))^{\top}+H A\,({\rm grad}\,\ln f)=0,
\end{cases}
\end{equation}
where $\xi$, $A$, and $H$ are the unit normal vector field, the
shape operator, and the mean curvature function  of
the hypersurface $\varphi(M)\subset (N^{m+1}, h)$ respectively, and the
operators $\Delta,\; {\rm grad}$ and $|,|$ are taken with respect to
the induced metric $g=\varphi^{*}h=\lambda^{2}{\bar g}$ on the
hypersurface.\\

Multiplying $f$ to both sides of each equation in the system (\ref{fbh}) we can rewrite the resulting system as (\ref{fbh2}).
\end{proof}

As a straightforward consequence, we have to following equations for  \fb \\hypersurfaces in Einstein spaces and space forms, which improves the equations given in \cite {PA}, and \cite{Ou4} respectively.
\begin{corollary}
A hypersurface $\phi:M^{m}\longrightarrow (N^{m+1}, h)$ in an Einstein space with ${\rm Ric}^N=\lambda h$  is  $f$-biharmonic if and only if it mean curvature function $H$ solves the following PDEs
\begin{equation}\label{E1}
\begin{cases}
\Delta (f\,H)-(f\,H )[|A|^{2}-\lambda]=0,\\
A\,({\rm grad}(f\,H)) +\frac{m}{2}(f\,H) {\rm grad}\, H
=0.
\end{cases}
\end{equation}
In particular, a hypersurface $\phi:M^{m}\longrightarrow N^{m+1}(C)$ in
a space form of constant sectional curvature $C$ is $f$-biharmonic if and only if its mean curvature function $H$ is a solution of
\begin{equation}\label{fbh3}
\begin{cases}
\Delta (f\,H)-(f\,H) [|A|^{2}-mC]=0,\\
A\,({\rm grad}(f\,H)) +\frac{m}{2}(f\,H){\rm grad}\, H
=0.
\end{cases}
\end{equation}
\end{corollary}

As the first application of the improved \fb hypersurface equation, we will prove the following theorem, which gives a rigidity result of \fb \\isometric immersions under some constraints of Ricci curvatures of the ambient space. Let us recall the following lemma, which will be used in proving our next theorem.\\

{\bf Lemma A.} (\cite{NU}) Assume that $(M, g)$ is a complete non-compact Riemannian manifold, and $\alpha$ is a non-negative smooth function on $M$. Then, every smooth $L^2$ function $u$ on $M$ satisfying the Schr$\ddot{\rm o}$dinger type equation
\begin{equation}\label{nu}
\Delta_g u = \alpha u	
\end{equation}
on $M$ is  a constant.\\
\begin{theorem}\label{Ric}
A complete hypersurface $\phi:M^{m}\longrightarrow (N^{m+1}, h)$ of a manifold $N$  with  ${\rm Ric}^N(\xi,\xi)\le |A|^2$ and the mean curvature vector field $\eta=H\xi$ satisfying $\int_M (fH)^2dv_g<\infty$ is $f$-biharmonic if and only if it is minimal, or
\begin{equation}\label{HP2}
\begin{cases}
{\rm Ric}^N(\xi,\xi)=|A|^{2},\\
({\rm Ric}^N\,(\xi))^{\top}=\frac{m}{2} {\rm grad}\, H,\;\;{\rm or}\;\; H=0.
\end{cases}
\end{equation}
In particular, a complete \fb hypersurface $\phi:M^{m}\longrightarrow (N^{m+1}, h)$ of a manifold $N$  of nonpositive Ricci curvature  and the mean curvature function satisfying $\int_M (fH)^2dv_g<\infty$ is minimal.
\end{theorem}
\begin{proof}
A straightforward computation yields
\begin{eqnarray}\label{GD3}
\Delta (fH)^2= 2fH\Delta (fH)+2 |\nabla (fH)|^2.
\end{eqnarray}

If the biharmonic hypersurface $\phi:M^{m}\longrightarrow (N^{m+1}, h)$ is compact, then, by Theorem \ref{FBH} we have Equation (\ref{fbh2}). Substituting the first equation of (\ref{fbh2}) into (\ref{GD3}) we have
\begin{eqnarray}\label{GD30}
\Delta (fH)^2= 2(fH)^2[|A|^2-{\rm Ric}^N(\xi,\xi)]+2 |\nabla (fH)|^2\ge 0
\end{eqnarray}
by the assumption that ${\rm Ric}^N(\xi,\xi)\le |A|^2$. It follows that the function $fH$ is a subharmonic function on a compact manifold, and hence it is constant.\\

If the biharmonic hypersurface $\phi:M^{m}\longrightarrow (N^{m+1}, h)$ is non-compact but complete, then the first equation of (\ref{fbh2}) implies that the function $fH$ is a solution of (\ref{nu}) with $\alpha=|A|^2-{\rm Ric}^N(\xi,\xi)$ being nonnegative. Using Lemma A we conclude that $fH$ is constant. So, in either case, the function $fH=C$, a constant. Substituting this into (\ref{fbh2}) we obtain the first statement of the theorem.\\

To prove the second statement, notice that $fH={\rm constant}$ and $f$ being positive imply that  we either have (i) $H\equiv 0$ in which case the hypersurface is minimal, or (ii) $H$ is never zero. If case (ii) happens, then we use $fH={\rm constant}\ne 0$ and the first equation \fb hypersurface equation to conclude that $0\le |A|^2={\rm Ric}^N(\xi,\xi)\le 0$. This implies that $A=0$ and hence the hypersurface is totally geodesic. It follows that $H\equiv 0$, which contradicts the assumption that $H$ is never zero. The contradiction shows that we only have case (i), which gives the second statement.
\end{proof}
From the proof of Theorem \ref{Ric} we have the following
\begin{corollary}
(i) A compact hypersurface $\phi:M^{m}\longrightarrow S^{m+1}$ with squared norm of the second fundamental form $|A|^2\ge m$ is $f$-biharmonic if and only if  it is biharmonic with  constant mean curvature and $|A|^2=m$.
(ii)  a complete hypersurface $\phi:M^{m}\longrightarrow S^{m+1}$ of the sphere $S^{m+1}$  with the mean curvature vector field $\eta=H\xi$ satisfying $ |A|^2\ge m$ and $\int_M (fH)^2dv_g<\infty$ is $f$-biharmonic if and only if the hypersurface is a biharmonic hypersurface of constant mean curvature and $|A|^2=m$.
\end{corollary}
As the second application of the improved \fb hypersurface equation, we give some classifications of \fb hypersurfaces in Einstein spaces and space forms.
\begin{theorem}
A compact  hypersurface of an Einstein space with  nonpositive scalar curvature is $f$-biharmonic if and only if it is a minimal hypersurface.
\end{theorem}
\begin{proof}
First of all, by Corollary \ref{23},  a minimal hypersurface $\phi:M^{m}\longrightarrow (N^{m+1}, h)$  is  $f$-biharmonic for any positive function $f$ on $M$. Conversely, let $\phi:M^{m}\longrightarrow (N^{m+1}, h)$ be a compact $f$-biharmonic hypersurface with the mean curvature vector field $\eta=H\xi$ on an Einstein manifold. Then, one can easily compute that
\begin{eqnarray}\label{NP}
\Delta((fH)^2)&=& 2(fH)\Delta(fH)+2|{\rm grad}(fH)|^2.
\end{eqnarray}
Using the first  equation of  (\ref{E1}) with $\lambda= \frac{{\rm R}_N}{m+1}$ and the assumption  on the scalar curvature ${\rm R}_N$ we have
\begin{eqnarray}\label{NP1}
\Delta((fH)^2)= 2(fH)^2[|A|^{2}-\frac{{\rm R}_N}{m+1}]+2|{\rm grad}(fH)|^2\ge 0.
\end{eqnarray}
It follows that $(fH)^2$ is a subharmonic function on a compact manifold and hence it is a constant. This implies that $fH=C$, a constant, from which and the fact the $f$ is a positive function we conclude that either (i) $H\equiv 0$, in which case the hypersurface in minimal, or (ii) $H(p)\ne 0$ for any $p\in M^m$. Now if case (ii) happens, then we $fH=C\ne 0$ and the second equation of the $f$-biharmonic equation of Einstein space to conclude that  $H$ is a nonzero constant. Using this and the first equation of (\ref{E1}) again we have $ 0\le |A|^2=\frac{{\rm R}_N}{m+1}\le 0$ and hence $A=0$. This means the hypersurface is actually totally geodesic and hence $H\equiv 0$, a contradiction to the assumption that $H(p)\ne 0$ for any $p\in M^m$. The contradiction shows that case (ii) cannot happen. Thus, we obtain the theorem.
\end{proof}

\begin{theorem}\label{EA}
A totally umbilical hypersurface $\phi:M^{m}\longrightarrow (N^{m+1}, h)$ in an Einstein space with ${\rm Ric}^N=\lambda h$  is  $f$-biharmonic if and only if it is totally geodesic or a biharmonic hypersurface with constant mean curvature $H=\pm \sqrt{\lambda/m}$. 
\end{theorem}
\begin{proof}
 If the hypersurface is totally umbilical, then we have $A(X)=HX$ for any vector tangent to the hypersurface and $|A|^2=mH^2$. Substituting these into (\ref{E1}) we have
\begin{equation}\label{E2}
\begin{cases}
\Delta (f\,H)-(f\,H )[mH^{2}-\lambda]=0,\\
H{\rm grad}(f\,H) +\frac{m}{2} (f\,H){\rm grad}\, H
=0.
\end{cases}
\end{equation}
On the other hand, it was proved in \cite{Ko} (see also Lemma 2.1 in \cite{JMS}) that a totally umbilical hypersurface in an Einstein space has constant mean curvature, i.e., $H={\rm constant}$.
If $H= 0$, then it is easily seen that the totally umbilical minimal hypersurface is actually totally geodesic. If $H={\rm constant}\ne 0$, then Equation of (\ref{E2}) is solved by $f={\rm constant}$ and

$mH^{2}-\lambda=0$. This implies that the totally umbilical hypersurface is a biharmonic hypersurafce with constant mean curvature $H=\pm \sqrt{\lambda/m}$. Summarizing the results in the two cases discussed above we obtain the theorem.
\end{proof}

As a consequence of Theorem \ref{EA}, we have
\begin{corollary}
A totally umbilical $f$-biharmonic hypersurface in Euclidean space $\mathbb{R}^{m+1}$ or hyperbolic space $H^{m+1}$ is a totally geodesic hypersurface, any totally umbilical $f$-biharmonic hypersurface in sphere $S^{m+1}$ is, up to isometries,  a part of the great sphere $S^m \subset S^{m+1}$ or a part of the small sphere  $S^m(\frac{1}{\sqrt{2}})$.
\end{corollary}
\begin{proof}
This follows from the well-known facts that space forms $\r^{m+1},\; H^{m+1}, S^{m+1}$ of constant sectional curvature $0, -1, 1$ are Einstein spaces with $\lambda= 0, -m, m$, and that totally umbilical biharmonic hypersurfaces in Euclidean spaces and hyperbolic spaces are totally geodesic, and the only totally umbilical biharmonic hypersurfaces in $S^{m+1}$ is a part of the  great sphere $S^m \subset S^{m+1}$ or a part of the small  $S^m(\frac{1}{\sqrt{2}})$ (\cite{CMO1}).
\end{proof}

It is well known that biharmonic hypersurfaces of Euclidean space $\mathbb{R}^{m+1}$ with at most two distinct principal curvature are minimal ones. The following example shows that this  is no longer true for $f$-biharmonic hypersurfaces.
\begin{example}
Let $\phi:D=\{ (\theta,x)\in (0,2\pi)\times \r ^{m-1}\}$ and  $\phi:D\longrightarrow (\r^{m+1}, h_0)$ with $\phi(\theta,x_1, \cdots, x_{m-1})=(R\cos\,\frac{\theta}{R}, R\sin\,\frac{\theta}{R}, x_1, \cdots, x_{m-1})$ be the isometric immersion into Euclidean space. One can check that it is a hypersurface  with two distinct principal curvatures, which is also an $f$-biharmonic hypersurface for any positive function $f$ from the family $f=C_1e^{x_1/R}+C_2e^{-x_1/R}$, where $C_1, C_2$ are constants.
\end{example}

In fact, we can take  
\begin{eqnarray}\notag
 && e_1={\rm d} \phi(\frac{\partial}{\partial \theta})=-\sin \frac{\theta}{R}\frac{\partial}{\partial y^1}+ \cos
\frac{\theta}{R}\frac{\partial}{\partial y^2},\\\notag
&&  e_i= {\rm d} \phi(\frac{\partial}{\partial x^{i-1}})=\frac{\partial}{\partial y^{1+i}},\; i= 2, \cdots, m,\\\notag
&&  \xi= \cos \frac{\theta}{R}\frac{\partial}{\partial y^1}+ \sin
\frac{\theta}{R}\frac{\partial}{\partial y^2}
\end{eqnarray}
as an orthonormal frame adapted to the hypersurface. Then a straightforward computation gives

\begin{eqnarray}\label{EJS}
\begin{cases}
Ae_1=-\frac{1}{R}e_1,\;\;Ae_i=0,\; \forall \; i=2, 3, \cdots, m.\\
H=\frac{1}{m}\sum_{i=1}^m\langle  Ae_i,e_i\rangle=-\frac{1}{mR}\ne 0\\
|A|^2=\sum_{i=1}^m|Ae_i|^2=\frac{1}{R^2},
\end{cases}
\end{eqnarray}
which show indeed that the hypersurface has two distinct principal curvatures and constant mean curvature.
Let $f: (0,2\pi)\times \r ^{m-1}\longrightarrow (0, \infty)$ be a function that does not depend on $\theta$, i.e., $f=f(x_1, \cdots, x_{m-1})$. Substituting this and (\ref{EJS}) into the \fb hypersurface equation (\ref{fbh3}) with $C=0$ we have 
\begin{equation}\label{TZ}
\Delta^{\r^{m-1}} f=\frac{1}{R^2}\,f.
\end{equation}
It follows that the cylinder $S^1\times \r^{m-1}\hookrightarrow \r^2\times \r^{m-1}$ is a an \fb hypersurface for any positive eigenfunction which solves equation (\ref{TZ}), in particular, any positive function from the family
$f=C_1e^{x_1/R}+C_2e^{-x_1/R}$, where $C_1, C_2$ are constants, is a solution of (\ref{TZ}).

\section{Biharmonic conformal immersions of surfaces into $3$-manifolds}

It is well known (See, e.g., \cite{Ta}) that a conformal immersion $\varphi : (M^2 ,{\bar g}) \longrightarrow (N^n ,h)$  with  $\varphi^{*}h=\lambda^2{\bar g}$ is harmonic if and only if the surface $\varphi(M^2) \hookrightarrow (N^n,h)$ is a minimal surface, i.e., the isometric immersion $\varphi(M^2) \hookrightarrow (N^n,h)$ is harmonic.  The following corollary, which follows from Theorem A, can be viewed as a generalization of this well-known result.
\begin{corollary}\label{Co41}
A conformal immersion $\varphi : (M^2 ,{\bar g}) \longrightarrow (N^n ,h)$  with  $\varphi^{*}h=\lambda^2{\bar g}$ is biharmonic if and only if the surface (i.e., the isometric immersion) $\varphi(M^2) \hookrightarrow (N^n,h)$ is an \fb surface with $f=\lambda^2$.
\end{corollary}

It was proved in \cite{Ou1} that a  hypersurface, i.e., an isometric immersion $M^m\hookrightarrow (N^{m+1}, h)$ with mean curvature $H$ and the shape operator $A$ is biharmonic if and only if
\begin{equation}\label{BHS}
\begin{cases}
\Delta H-H |A|^{2}+H{\rm
Ric}^N(\xi,\xi)=0,\\
 2A\,({\rm grad}\,H) +\frac{m}{2} {\rm grad}\, H^2
-2\, H \,({\rm Ric}^N\,(\xi))^{\top}=0,
\end{cases}
\end{equation}
where ${\rm Ric}^N : T_qN\longrightarrow T_qN$ denotes the Ricci
operator of the ambient space defined by $\langle {\rm Ric}^N\, (Z),
W\rangle={\rm Ric}^N (Z, W)$,  $\xi$ is the unit normal vector field, and $\Delta$ and ${\rm grad}$ denote the Laplace and the gradient operators defined by the induced metric on the hypersurface.\\

Later in \cite{Ou3}, it was proved that a conformal immersion $\varphi:(M^2, {\bar g})\longrightarrow (N^3, h)$ with $\varphi^{*}h=\lambda^2{\bar g}$ is biharmonic if and only if
\begin{equation}\label{BCI3}
\begin{cases}
\Delta H-H[ |A|^{2}-{\rm
Ric}^N(\xi,\xi)-(\Delta \lambda^2)/\lambda^2]+4g({\rm grad} \ln \lambda), {\rm grad} H)=0,\\
A\,({\rm grad}\,H) +H [{\rm grad}\, H
-({\rm Ric}^N\,(\xi))^{\top}]+2 A\,({\rm grad}\,\ln \lambda)]=0,
\end{cases}
\end{equation}
where $\xi$, $A$, and $H$ are the unit normal vector field, the
shape operator, and the mean curvature function  of
the surface $\varphi(M)\subset (N^{3}, h)$ respectively, and the
operators $\Delta,\; {\rm grad}$ and $|,|$ are taken with respect to
the induced metric $g=\varphi^{*}h=\lambda^2{\bar g}$ on the surface.\\

Our next theorem gives an improved equation for biharmonic conformal immersions of surfaces into a $3$-manifold by writing it in a  form similar to biharmonic surface equation (\ref{BHS}),  which  turns out to be very useful in proving some classification and/or  existence results on biharmonic conformal surfaces.\\

\begin{theorem}\label{NEW}
A conformal immersion
\begin{equation}\label{MAP1}
\varphi : (M^{2},{\bar g}) \longrightarrow
(N^3,h)
\end{equation}
 into a $3$-dimensional manifold with
$\varphi^{*}h=\lambda^2{\bar g}$ is biharmonic if and only if the surface $\varphi(M^2)\hookrightarrow (N^3, h)$ is an $f$-biharmonic surface with $f=\lambda^2$, which is characterized by the equation
\begin{equation}\label{M03}
\begin{cases}
\Delta (\lambda^2 H )-(\lambda^2H)[|A|^2-{\rm Ric}^N(\xi,\xi)]=0,\\A({\rm grad} (\lambda^2 H))+ (\lambda^2 H) [{\rm grad}
H- \,({\rm Ric}^N\,(\xi))^{\top}]=0.
\end{cases}
\end{equation}
where $\xi$, $A$, and $H$ are the unit normal vector field, the
shape operator, and the mean curvature function  of
the surface $\varphi(M)\subset (N^3, h)$ respectively, and the
operators $\Delta,\; {\rm grad}$ and $|,|$ are taken with respect to
the induced metric $g=\varphi^{*}h=f{\bar g}$ on the
surface.
\end{theorem}
\begin{proof}
One can easily see that a map $\varphi : (M^{2},{\bar g}) \longrightarrow
(N^3,h)$ is a conformal immersion with conformal factor $\lambda$, i.e., $\varphi^{*}h=\lambda^2{\bar g}$, if and only if the map
\begin{equation} \label{IDI}
\varphi : (M^{2}, g=\lambda^2{\bar g}) \longrightarrow
(N^3,h)
\end{equation}
is an isometric immersion. By Corollary \ref{Co41}, the conformal immersion $\varphi : (M^{2}, {\bar g}) \longrightarrow
(N^3,h)$ is  biharmonic if and only if the surface $\varphi : (M^{2}, \lambda^2{\bar g}) \longrightarrow
(N^3,h)$ is an \fb surface. Applying Equation (\ref{fbh2}) with $m=2$ and $f=\lambda^2$ we obtain the biharmonic conformal equation (\ref{M03}) and hence the theorem.
\end{proof}
As an immediate consequence of Theorem \ref{NEW}, we have the following
\begin{corollary}\label{C2}
A conformal immersion $\varphi : (M^{2}, {\bar
g}) \longrightarrow (N^3(C),
h_0)$ into $3$-dimensional space of constant sectional curvature $C$
with $\varphi^{*}h_{0}=\lambda^2{\bar
g}$ is biharmonic if and only if
\begin{equation}\label{C}
\begin{cases}
\Delta (\lambda^2 H )-(\lambda^2H)(|A|^2-2C)=0,\\A({\rm grad} (\lambda^2 H))+ (\lambda^2 H) {\rm grad}
H=0.\\
\end{cases}
\end{equation}
where  $A$ and $H$ are the shape operator and the mean curvature function of the surface
respectively, and the operators $\Delta,\; {\rm grad}$ and $|,|$ are
taken with respect to the induced metric
$g=\varphi^{*}h=\lambda^2{\bar g}$ on the surface.
\end{corollary}

\begin{theorem}\label{MT2}
Let $\varphi : (M^{2},{\bar g}) \longrightarrow (\r^3, h_0)$ be a
 biharmonic conformal immersion from a complete surface into
$3$-dimensional Euclidean space with
$\varphi^{*}h_0=\lambda^2{\bar g}$. If the mean curvature $H$ of the surface $\varphi : (M^{2}, g=\lambda^2{\bar g}) \longrightarrow (\r^3, h_0)$ satisfies $\int_M\,\lambda^4H^2 \,dv_{g}<\infty$.  Then the biharmonic conformal immersion $\varphi$ is a minimal immersion.
\end{theorem}

\begin{proof}
Note that $\varphi : (M^{2},{\bar g}) \longrightarrow (\r^3, h_0)$ is a biharmonic conformal immersion implies that the surface $ (M^{2}, g=\lambda^2{\bar g}) $ can be biharmonically conformally immersed into Euclidean space $ (\r^3, h_0)$. By Corollary \ref{C2}, we have
\begin{equation}\label{M310}
\begin{cases}
\Delta (\lambda^2 H )=(\lambda^2 H) |A|^2,\\A({\rm grad} (\lambda^2  H))+ (\lambda^2 H) {\rm grad}
H=0.\\
\end{cases}
\end{equation}
where  $A$, and $H$ are  the shape operator, and the mean curvature
function of the surface  $\varphi(M)\subset (\r^3, h_0)$
respectively, and the operators $\Delta,\; {\rm grad}$ and $|,|$ are
taken with respect to the induced metric $g=\varphi^{*}h=\lambda^2 {\bar g}$ on $M$.\\

From  the assumption $\int_M\,\lambda^4 H^2 \,dv_{g}<\infty$ and the first equation of (\ref{M310}) we see that the function $\lambda^2 H\in L^2(M)$ is a solution of the Schr$\ddot{\rm o}$dinger type equation
\begin{equation}
\Delta u = \alpha u	
\end{equation} 
for a fixed nonnegative function $\alpha=|A|^2$ on a complete manifold $(M^2, g)$.
By Lemma A (Section 3), we conclude that $\lambda^2H=C$,  a constant. Substituting this into the second equation of (\ref{M310}) we have either $H=0$ which means the conformal immersion is minimal, or $H=C_1$ a constant. If $H=C_1\ne 0$, then $\lambda^2H=C$ implies that $\lambda^2=C/H=C/C_1$ is a constant.  In this case, the conformal immersion is a homothetic immersion. By a well-known result of Chen and Jiang, any biharmonic homothetic immersion of a surface into $\r^3$ is minimal. Therefore, in either case, we conclude that the conformal biharmonic immersion is a minimal homothetic immersion.
\end{proof}

Motivated by the beautiful theory of minimal surfaces as conformal harmonic immersions and the attempt to understand the relationship among biharmonicity, conformality and $f$-biharmonicity of maps from surfaces, we studied biharmonic conformal immersions of surfaces in \cite{Ou2}, and \cite{Ou3}. A Fundamental question is: what are the geometric and/or topological obstacles for a surface to admit a biharmonic conformal immersion into a ``nice space" like a space form?  Recall (See \cite{Ou3}) that a surface in a Riemannian manifold $(N^{3}, h)$ defined by an isometric immersion $\varphi: (M^{2}, g) \longrightarrow (N^{3},h)$ is said to admit a {\em biharmonic conformal immersion} into $(N^{3}, h)$, if there exists  a function $f :M^2\longrightarrow \r^{+}$ such that the conformal immersion $\varphi: (M^{2}, {\bar g}=f^{-1}g) \longrightarrow (N^{3},h)$ with conformal factor $f$ is a biharmonic map. In such a case, we also say that the surface $\varphi: (M^{2}, g) \longrightarrow (N^{3},h)$  can be {\em biharmonically conformally immersed} into $(N^{3}, h)$. Later in studying $f$-biharmonic maps and \fb submanifolds in \cite{Ou4}, Theorem A was proved. It follows from Theorem A that a surface  that $\varphi: (M^{2}, g) \longrightarrow (N^{3},h)$ that admits a  biharmonic conformal immersion into $(N^{3}, h)$ if and only if the surface is an \fb surface.\\

Examples of surfaces that can be biharmonically conformally immersed into a $3$-manifold include the following

\begin{itemize}
\item For any biharmonic conformal immersion  $\varphi: (M^{2}, {\bar g}) \longrightarrow (N^{3}, h)$ with $\varphi^{*}h=\lambda^2{\bar g}$, its associated surface  $\varphi: (M^{2}, g=\varphi^{*}h) \longrightarrow (N^{3}, h)$ admits a biharmonic conformal immersion into $(N^{3}, h)$ as we see that there exists $f=\lambda^2$ such that $f^{-1}(\varphi^{*}h)= {\bar g}$. For examples of biharmonic conformal immersions of surfaces, see \cite{Ou2} and \cite{Ou3}.
\item Any minimal surface (i.e., harmonic isometric immersion) $\varphi:M^2\longrightarrow (N^3,h)$ admits a biharmonic conformal immersion into $(N^3,h)$. This is because, by Corollary \ref{23}, a minimal surface $\varphi:M^2\longrightarrow (N^3,h)$ is \fb for any positive function $f$ defined on the surface. 
\end{itemize}

Using the improved equation for  biharmonic conformal immersions of surfaces we can prove the following
\begin{theorem}
If a compact surface $M^2\hookrightarrow S^3$ with the squared norm of the second fundamental form $|A|^2\ge 2$ can be biharmonically conformally immersed into $S^3$, then  $M^2$ is minimal, or up to a homothetic change of the metric, $M^2= S^2(\frac{1}{2})$.
\end{theorem}

\begin{proof}
Suppose the surface is given by an isometric immersion $\phi:M^2\longrightarrow (S^3,h)$, where $h$ denotes the standard metric on $S^3$ with constant sectional curvature $1$. By definition, if the surface admits a biharmonic conformal immersion into $S^3$, then there exists a function $\lambda: M\longrightarrow (0, \infty)$ such that the conformal immersion $\phi:(M^2, {\bar g}=\lambda^2 g)\longrightarrow (S^3,h)$ is biharmonic. So,  by (\ref{C}), we have 
\begin{equation}\label{S3}
\begin{cases}
\Delta (\lambda^2 H )-(\lambda^2H)(|A|^2-2)=0,\\A({\rm grad} (\lambda^2 H))+ (\lambda^2 H) {\rm grad}
H=0.\\
\end{cases}
\end{equation}
We compute
\begin{eqnarray}\notag
\Delta ((\lambda^2 H )^2)&=&2(\lambda^2 H )\Delta (\lambda^2 H )+2|{\rm grad}  (\lambda^2 H )^2)|^2\\\label{La}
&=&2(\lambda^2 H )^2(|A|^2-2)+2|{\rm grad}  (\lambda^2 H )^2)|^2.
\end{eqnarray}
where in obtaining the second equality we have used the first equation of (\ref{S3}).  From Equation (\ref{La})  and the assumption that $|A|^2\ge 2$ we conclude that $\Delta ((\lambda^2 H )^2)\ge 0$, i.e.,  the function $\lambda^2 H$ is a subharmonic function on the compact manifold $M$. It follows that $\lambda^2 H$ is a constant function. Using this and the second equation of (\ref{S3}) we see that $H$ must be a constant. If $H=0$, then the surface $\phi:M^2\longrightarrow (S^3,h)$ is minimal. Otherwise, if $H=C\ne 0$,  then $\lambda$ is also a constant and hence, up to a homothety, the compact surface $\phi:M^2\longrightarrow (S^3,h)$ is biharmonic. It follows from a well-known result (\cite{CMO1}) of biharmonic surfaces in $S^3$ that $M^2= S^2(\frac{1}{2})$. Thus, we obtain the theorem.
\end{proof}

\end{document}